\documentclass[10pt]{amsart}
\usepackage{amssymb,MnSymbol}
\usepackage{amsthm,amsmath}
\usepackage{stmaryrd}

\title[Axiomatizations of quasi-Lov\'asz extensions]{Axiomatizations of quasi-Lov\'asz extensions of pseudo-Boolean functions}

\author{Miguel Couceiro}
\address{Mathematics Research Unit, FSTC, University of Luxembourg \\
6, rue Coudenhove-Kalergi, L-1359 Luxembourg, Luxembourg}%
\email{miguel.couceiro[at]uni.lu }

\author{Jean-Luc Marichal}
\address{Mathematics Research Unit, FSTC, University of Luxembourg \\
6, rue Coudenhove-Kalergi, L-1359 Luxembourg, Luxembourg}%
\email{jean-luc.marichal[at]uni.lu }

\date{June 22, 2011}

\begin{document}

\theoremstyle{plain}
\newtheorem{theorem}{Theorem}
\newtheorem{lemma}[theorem]{Lemma}
\newtheorem{proposition}[theorem]{Proposition}
\newtheorem{corollary}[theorem]{Corollary}
\newtheorem{fact}[theorem]{Fact}
\newtheorem*{main}{Main Theorem}

\theoremstyle{definition}
\newtheorem{definition}[theorem]{Definition}
\newtheorem{example}[theorem]{Example}

\theoremstyle{remark}
\newtheorem*{conjecture}{Conjecture}
\newtheorem{remark}{Remark}
\newtheorem{claim}{Claim}

\newcommand{\N}{\mathbb{N}}
\newcommand{\R}{\mathbb{R}}
\newcommand{\B}{\mathbb{B}}
\newcommand{\Vspace}{\vspace{2ex}}
\newcommand{\bfx}{\mathbf{x}}

\begin{abstract}
We introduce the concept of quasi-Lov\'asz extension as being a mapping $f\colon I^n\to\R$ defined on a nonempty real interval $I$ containing
the origin and which can be factorized as $f(x_1,\ldots,x_n)=L(\varphi(x_1),\ldots,\varphi(x_n))$, where $L$ is the Lov\'asz extension of a
pseudo-Boolean function $\psi\colon\{0,1\}^n\to\R$ (i.e., the function $L\colon\R^n\to\R$ whose restriction to each simplex of the standard
triangulation of $[0,1]^n$ is the unique affine function which agrees with $\psi$ at the vertices of this simplex) and $\varphi\colon I\to\R$ is
a nondecreasing function vanishing at the origin. These functions appear naturally within the scope of decision making under uncertainty since
they subsume overall preference functionals associated with discrete Choquet integrals whose variables are transformed by a given utility
function. To axiomatize the class of quasi-Lov\'asz extensions, we propose generalizations of properties used to characterize the Lov\'asz
extensions, including a comonotonic version of modularity and a natural relaxation of homogeneity. A variant of the latter property enables us
to axiomatize also the class of symmetric quasi-Lov\'asz extensions, which are compositions of symmetric Lov\'asz extensions with $1$-place
nondecreasing odd functions.
\end{abstract}



\maketitle

\section{Introduction}

Aggregation functions arise wherever merging information is needed: applied and pure mathematics (probability, statistics, decision theory,
functional equations), operations research, computer science, and many applied fields (economics and finance, pattern recognition and image
processing, data fusion, etc.). For recent references, see Beliakov et al.~\cite{BelPraCal07} and Grabisch et al.~\cite{GraMarMesPap09}.

The discrete Choquet integral has been widely investigated in aggregation theory due to its many applications, for instance, in decision making
(see the edited book \cite{GraMurSug00}). A convenient way to introduce the discrete Choquet integral is via the concept of Lov\'asz extension.
An $n$-place Lov\'asz extension is a continuous function $L\colon\R^n\to\R$ whose restriction to each of the $n!$ subdomains
$$
\R^n_{\sigma}=\{\bfx=(x_1,\ldots,x_n)\in\R^n : x_{\sigma(1)}\leqslant\cdots\leqslant x_{\sigma(n)}\},\qquad \sigma\in S_n,
$$
is an affine function, where $S_n$ denotes the set of permutations on $[n]=\{1,\ldots,n\}$. An $n$-place Choquet integral is simply a
nondecreasing (in each variable) $n$-place Lov\'asz extension which vanishes at the origin. For general background, see
\cite[{\S}5.4]{GraMarMesPap09}.

The class of $n$-place Lov\'asz extensions has been axiomatized by the authors \cite{CouMar} by means of two noteworthy aggregation properties,
namely comonotonic additivity and horizontal min-additivity (for earlier axiomatizations of the $n$-place Choquet integrals, see, e.g.,
\cite{BenMesViv02,deCBol92}). Recall that a function $f\colon\R^{n}\to\R$ is said to be \emph{comonotonically additive} if, for every $\sigma\in
S_n$, we have
$$
f(\bfx+\bfx') = f(\bfx)+ f(\bfx'),\qquad \bfx,\bfx'\in\R^n_{\sigma}.
$$
The function $f$ is said to be \emph{horizontally min-additive} if
$$
f(\bfx) = f(\bfx\wedge c)+ f(\bfx-(\bfx\wedge c)),\qquad \bfx\in\R^n,~c\in\R,
$$
where $\bfx\wedge c$ denotes the $n$-tuple whose $i$th component is $x_i\wedge c=\min(x_i,c)$.

In this paper we consider a generalization of Lov\'asz extensions, which we call quasi-Lov\'asz extensions, and which are best described by the
following equation
$$
f(x_1,\ldots,x_n)=L(\varphi(x_1),\ldots,\varphi(x_n))
$$
where $L$ is a Lov\'asz extension and $\varphi$ a nondecreasing function such that $\varphi(0)=0$. Such an aggregation function is used in
decision under uncertainty, where $\varphi$ is a utility function and $f$ an overall preference functional. It is also used in multi-criteria
decision making where the criteria are commensurate (i.e., expressed in a common scale). For a recent reference, see Bouyssou et
al.~\cite{BouDubPraPir09}.

To axiomatize the class of quasi-Lov\'asz extensions, we propose the following generalizations of comonotonic additivity and horizontal
min-additivity, namely comonotonic modularity and invariance under horizontal min-differences (as well as its dual counterpart), which we now
briefly describe. We say that a function $f\colon\R^{n}\to\R$ is \emph{comonotonically modular} if, for every $\sigma\in S_n$, we have
$$
f(\bfx)+ f(\bfx')=f(\bfx\wedge\bfx')+f(\bfx\vee\bfx'),\qquad \bfx,\bfx'\in\R^n_{\sigma},
$$
where $\bfx\wedge\bfx'$ (resp.\ $\bfx\vee\bfx'$) denotes the $n$-tuple whose $i$th component is $x_i\wedge x'_i=\min(x_i,x'_i)$ (resp.\ $x_i\vee
x'_i=\max(x_i,x'_i)$). We say that $f$ is \emph{invariant under horizontal min-differences} if
$$
f(\bfx) - f(\bfx\wedge c) = f([\bfx]_c) - f([\bfx]_c\wedge c),\qquad \bfx\in\R^n,~c\in\R,
$$
where $[\bfx]_c$ denotes the $n$-tuple whose $i$th component is $0$, if $x_i\leqslant c$, and $x_i$, otherwise.

The outline of this paper is as follows. In Section 2 we recall the definitions of Lov\'asz extensions, discrete Choquet integrals, as well as
their symmetric versions, and present representations for these functions. In Section 3 we define the concept of quasi-Lov\'asz extension and
its symmetric version, introduce natural relaxations of homogeneity, namely weak homogeneity and odd homogeneity, and characterize those
quasi-Lov\'asz extensions (resp.\ symmetric quasi-Lov\'asz extensions) that are weakly homogeneous (resp.\ oddly homogeneous). In Section 4 we
define the concepts of comonotonic modularity, invariance under horizontal min-differences and invariance under horizontal max-differences, and
completely describe the function classes axiomatized by each of these properties. In Section 5 we give axiomatizations of the class of
quasi-Lov\'asz extensions by means of the properties above and describe all possible factorizations of quasi-Lov\'asz extensions into
compositions of Lov\'asz extensions with $1$-place functions. In Section 6 we present analogous results for the symmetric
quasi-Lov\'asz extensions. Finally, in Section 7 we show that the so-called quasi-polynomial functions \cite{CouMar09} on closed intervals form a noteworthy subclass of comonotonically modular functions.

We employ the following notation throughout the paper. Let $\B=\{0,1\}$, $\R_+=\left[0,+\infty\right[$, and $\R_-=\left]-\infty,0\right]$. The
symbol $I$ denotes a nonempty real interval, possibly unbounded, containing $0$. We also introduce the notation $I_+=I\cap\R_+$,
$I_-=I\cap\R_-$, and $I^n_{\sigma}=I^n\cap\R^n_{\sigma}$. A function $f\colon I^n\to\R$, where $I$ is centered at $0$, is said to be \emph{odd}
if $f(-\bfx)=-f(\bfx)$. For any function $f\colon I^n\to\R$, we define $f_0=f-f(\mathbf{0})$. For every $A\subseteq [n]$, the symbol
$\mathbf{1}_A$ denotes the $n$-tuple whose $i$th component is $1$, if $i\in A$, and $0$, otherwise. Let also $\mathbf{1}=\mathbf{1}_{[n]}$ and
$\mathbf{0}=\mathbf{1}_{\varnothing}$. The symbols $\wedge$ and $\vee$ denote the minimum and maximum functions, respectively. For every
$\bfx\in\R^n$, let $\bfx^+=\bfx\vee 0$ and $\bfx^-=(-\bfx)^+$. For every $\bfx\in\R^n$ and every $c\in\R_+$ (resp.\ $c\in\R_-$) we denote by
$[\bfx]_c$ (resp.\ $[\bfx]^c$) the $n$-tuple whose $i$th component is $0$, if $x_i\leqslant c$ (resp.\ $x_i\geqslant c$), and $x_i$, otherwise.

In order not to restrict our framework to functions defined on $\R$, we consider functions defined on intervals $I$ containing $0$, in
particular of the forms $I_+$, $I_-$, and those centered at $0$.

\section{Lov\'asz extensions and symmetric Lov\'asz extensions}

We now recall the concepts of Lov\'asz extension and symmetric Lov\'asz extension.

Consider an $n$-place \emph{pseudo-Boolean function}, i.e.\ a function $\psi\colon\B^n\to\R$, and define the set function $v_{\psi}\colon
2^{[n]}\to\R$ by $v_{\psi}(A)=\psi(\mathbf{1}_A)$ for every $A\subseteq [n]$. Hammer and Rudeanu~\cite{HamRud68} showed that such a function has
a unique representation as a multilinear polynomial of $n$ variables
$$
\psi(\bfx)=\sum_{A\subseteq [n]}a_{\psi}(A)\,\prod_{i\in A}x_i\, ,
$$
where the set function $a_{\psi}\colon 2^{[n]}\to\R$, called the \emph{M\"obius transform} of $v_{\psi}$, is defined by
$$
a_{\psi}(A)=\sum_{B\subseteq A}(-1)^{|A|-|B|}\, v_{\psi}(B).
$$

The \emph{Lov\'asz extension} of a pseudo-Boolean function $\psi\colon\B^n\to\R$ is the function $L_{\psi}\colon\R^n\to\R$ whose restriction to
each subdomain $\R^n_{\sigma}$ $(\sigma\in S_n)$ is the unique affine function which agrees with $\psi$ at the $n+1$ vertices of the $n$-simplex
$[0,1]^n\cap\R^n_{\sigma}$ (see \cite{Lov83,Sin84}). We then have $L_{\psi}|_{\B^n}=\psi$.

It can be shown (see \cite[{\S}5.4.2]{GraMarMesPap09}) that the Lov\'asz extension of a pseudo-Boolean function $\psi\colon\B^n\to\R$ is the
continuous function
\begin{equation}\label{eq:sdfaf678df}
L_{\psi}(\bfx)=\sum_{A\subseteq [n]}a_{\psi}(A)\,\bigwedge_{i\in A}x_i\, ,\qquad \bfx\in\R^n.
\end{equation}
Its restriction to $\R^n_{\sigma}$ is the affine function
\begin{equation}\label{eq:sdfaf678}
L_{\psi}(\bfx) = \psi(\mathbf{0})+\sum_{i\in [n]}
x_{\sigma(i)}\,\big(v_{\psi}(A_{\sigma}^{\uparrow}(i))-v_{\psi}(A_{\sigma}^{\uparrow}(i+1))\big),\qquad \bfx\in\R^n_{\sigma},
\end{equation}
or equivalently,
\begin{equation}\label{eq:sdfaf678dew}
L_{\psi}(\bfx) = \psi(\mathbf{0})+\sum_{i\in [n]}
x_{\sigma(i)}\,\big(L_{\psi}(\mathbf{1}_{A_{\sigma}^{\uparrow}(i)})-L_{\psi}(\mathbf{1}_{A_{\sigma}^{\uparrow}(i+1)})\big),\qquad
\bfx\in\R^n_{\sigma},
\end{equation}
where $A_{\sigma}^{\uparrow}(i)=\{\sigma(i),\ldots,\sigma(n)\}$, with the convention that $A_{\sigma}^{\uparrow}(n+1)=\varnothing$. Indeed, for
any $k\in [n+1]$, both sides of each of the equations (\ref{eq:sdfaf678}) and (\ref{eq:sdfaf678dew}) agree at
$\bfx=\mathbf{1}_{A_{\sigma}^{\uparrow}(k)}$.

It is noteworthy that $L_{\psi}$ can also be represented by
\begin{equation}\label{eq:sdfaf678dew2}
L_{\psi}(\bfx) = \psi(\mathbf{0})+\sum_{i\in [n]}
x_{\sigma(i)}\,\big(L_{\psi}(-\mathbf{1}_{A_{\sigma}^{\downarrow}(i-1)})-L_{\psi}(-\mathbf{1}_{A_{\sigma}^{\downarrow}(i)})\big),\qquad
\bfx\in\R^n_{\sigma},
\end{equation}
where $A_{\sigma}^{\downarrow}(i)=\{\sigma(1),\ldots,\sigma(i)\}$, with the convention that $A_{\sigma}^{\downarrow}(0)=\varnothing$. Indeed,
for any $k\in [n+1]$, by (\ref{eq:sdfaf678dew}) we have
$$
L_{\psi}(-\mathbf{1}_{A_{\sigma}^{\downarrow}(k-1)})=\psi(\mathbf{0})+L_{\psi}(\mathbf{1}_{A_{\sigma}^{\uparrow}(k)})-L_{\psi}(\mathbf{1}_{A_{\sigma}^{\uparrow}(1)}).
$$

Let $\psi^d$ denotes the \emph{dual} of $\psi$, that is the function $\psi^d\colon\B^n\to\R$ defined by
$\psi^d(\bfx)=\psi(\mathbf{0})+\psi(\mathbf{1})-\psi(\mathbf{1}-\bfx)$. The next result provides further representations for $L_{\psi}$.

\begin{proposition}
The Lov\'asz extension of a pseudo-Boolean function $\psi\colon\B^n\to\R$ is given by
\begin{equation}\label{eq:sdfaf678dew223}
L_{\psi}(\bfx)=\psi(\mathbf{0})+\sum_{A\subseteq [n]}a_{\psi^d}(A)\,\bigvee_{i\in A}x_i\, ,
\end{equation}
and
\begin{equation}\label{eq:sdfaf678dew234}
L_{\psi}(\bfx)=\psi(\mathbf{0})+L_{\psi}(\bfx^+)-L_{\psi^d}(\bfx^-).
\end{equation}
\end{proposition}

\begin{proof}
Since the Lov\'asz extension $L_{\psi}$ is additive with respect to its restriction $\psi$ (i.e., $L_{\psi+\psi'}=L_{\psi}+L_{\psi'}$), for
every $\bfx\in\R^n$, we have
$$
L_{\psi}(\bfx) = \psi(\mathbf{0})+\psi(\mathbf{1})-L_{\psi^d}(\mathbf{1}-\bfx) =
\psi(\mathbf{0})+\psi^d(\mathbf{1})-L_{\psi^d}(\mathbf{1}-\bfx),
$$
that is, by using (\ref{eq:sdfaf678df}),
$$
L_{\psi}(\bfx) = \psi(\mathbf{0})+\sum_{A\subseteq [n]}a_{\psi^d}(A)-\sum_{A\subseteq [n]}a_{\psi^d}(A)\,\bigg(1-\bigvee_{i\in A}x_i\bigg),
$$
which proves (\ref{eq:sdfaf678dew223}).

Now, for every $A\subseteq [n]$, we have $\bigwedge_{i\in A}x_i=\bigwedge_{i\in A}x_i^++\bigwedge_{i\in A}(-x_i^-)$ and hence by
(\ref{eq:sdfaf678df}),
\begin{eqnarray*}
L_{\psi}(\bfx) &=& \psi(\mathbf{0})+\sum_{\varnothing\neq A\subseteq [n]}a_{\psi}(A)\,\Big(\bigwedge_{i\in A}x_i^++\bigwedge_{i\in
A}(-x_i^-)\Big)\\
&=& \psi(\mathbf{0})+\sum_{\varnothing\neq A\subseteq [n]}a_{\psi}(A)\,\bigwedge_{i\in A}x_i^+-\sum_{A\subseteq [n]}a_{\psi}(A)\,\bigvee_{i\in
A}x_i^-
\end{eqnarray*}
which, using (\ref{eq:sdfaf678dew223}) and the identity $\psi^{dd}=\psi$, leads to (\ref{eq:sdfaf678dew234}).
\end{proof}


A function $f\colon\R^n\to\R$ is said to be a \emph{Lov\'asz extension} if there is a pseudo-Boolean function $\psi\colon\B^n\to\R$ such that
$f=L_{\psi}$.

An $n$-place \emph{Choquet integral} is a nondecreasing Lov\'asz extension $L_{\psi}\colon\R^n\to\R$ such that $L_{\psi}(\mathbf{0})=0$. It is
easy to see that a Lov\'asz extension $L\colon\R^n\to\R$ is an $n$-place Choquet integral if and only if its underlying pseudo-Boolean function
$\psi=L|_{\B^n}$ is nondecreasing and vanishes at the origin (see \cite[{\S}5.4]{GraMarMesPap09}).

The \emph{symmetric Lov\'asz extension} of a pseudo-Boolean function $\psi\colon\B^n\to\R$ is the function $\check{L}\colon\R^n\to\R$ defined by
(see \cite{CouMar})
$$
\check{L}_{\psi}(\bfx)=\psi(\mathbf{0})+L_{\psi}(\bfx^+)-L_{\psi}(\bfx^-).
$$
In particular, we see that $\check{L}_{\psi}-\check{L}_{\psi}(\mathbf{0})=\check{L}_{\psi}-\psi(\mathbf{0})$ is an odd function.

It is easy to see that the restriction of $\check{L}_{\psi}$ to $\R^n_{\sigma}$ is the function
\begin{eqnarray}
\check{L}_{\psi}(\bfx) &=& \psi(\mathbf{0})+\sum_{1\leqslant i\leqslant p}
x_{\sigma(i)}\,\big(L_{\psi}(\mathbf{1}_{A_{\sigma}^{\downarrow}(i)})-L_{\psi}(\mathbf{1}_{A_{\sigma}^{\downarrow}(i-1)})\big)\nonumber\\
&& \null +\sum_{p+1\leqslant i\leqslant n}
x_{\sigma(i)}\,\big(L_{\psi}(\mathbf{1}_{A_{\sigma}^{\uparrow}(i)})-L_{\psi}(\mathbf{1}_{A_{\sigma}^{\uparrow}(i+1)})\big),\qquad
\bfx\in\R^n_{\sigma},\label{eq:sdfsfd65dsf}
\end{eqnarray}
where the integer $p\in\{0,\ldots,n\}$ is such that $x_{\sigma(p)}<0\leqslant x_{\sigma(p+1)}$.

A function $f\colon\R^n\to\R$ is said to be a \emph{symmetric Lov\'asz extension} if there is a pseudo-Boolean function $\psi\colon\B^n\to\R$
such that $f=\check{L}_{\psi}$.

Nondecreasing symmetric Lov\'asz extensions vanishing at the origin, also called \emph{discrete symmetric Choquet integrals}, were introduced by
\v{S}ipo\v{s} \cite{Sip79} (see also \cite[{\S}5.4]{GraMarMesPap09}).

\section{Quasi-Lov\'asz extensions and symmetric quasi-Lov\'asz extensions}

In this section we introduce the concepts of quasi-Lov\'asz extension and symmetric quasi-Lov\'asz extension. We also introduce natural
relaxations of homogeneity, namely weak homogeneity and odd homogeneity, and characterize those quasi-Lov\'asz extensions (resp.\ symmetric
quasi-Lov\'asz extensions) that are weakly homogeneous (resp.\ oddly homogeneous). Recall that $I$ is a real interval containing $0$.

A \emph{quasi-Lov\'asz extension} is a function $f\colon I^n\to\R$ defined by
$$
f=L\circ(\varphi,\ldots,\varphi),
$$
also written $f=L\circ\varphi$, where $L\colon\R^n\to\R$ is a Lov\'asz extension and $\varphi\colon I\to\R$ is a nondecreasing function
satisfying $\varphi(0)=0$. Observe that a function $f\colon I^n\to\R$ is a quasi-Lov\'asz extension if and only if $f_0=L_0\circ\varphi$.

\begin{lemma}\label{lemma:sdf7d}
Assume $I\subseteq\R_+$. For every quasi-Lov\'asz extension $f\colon I^n\to\R$, $f=L\circ\varphi$, we have
\begin{equation}\label{eq:dsa786}
f_0(x\mathbf{1}_A)=\varphi(x)L_0(\mathbf{1}_A),\qquad x\in I,~A\subseteq [n].
\end{equation}
\end{lemma}

\begin{proof}
For every $x\in I$ and every $A\subseteq [n]$, there exists $\sigma\in S_n$ such that $x\mathbf{1}_A\in I^n_{\sigma}$ and, using
(\ref{eq:sdfaf678dew}), we then obtain
$$
f_0(x\mathbf{1}_A) = \sum_{n-|A|+1\leqslant i\leqslant n}
\varphi(x)\,\big(L(\mathbf{1}_{A_{\sigma}^{\uparrow}(i)})-L(\mathbf{1}_{A_{\sigma}^{\uparrow}(i+1)})\big)= \varphi(x)L_0(\mathbf{1}_A).\qedhere
$$
\end{proof}

Observe that if $[0,1]\subseteq I\subseteq\R_+$ and $\varphi(1)=1$, then the equation in (\ref{eq:dsa786}) becomes
$f_0(x\mathbf{1}_A)=\varphi(x)f_0(\mathbf{1}_A)$. This motivates the following definition. We say that a function $f\colon I^n\to\R$, where
$I\subseteq\R_+$, is \emph{weakly homogeneous} if there exists a nondecreasing function $\varphi\colon I\to\R$ satisfying $\varphi(0)=0$ such
that $f(x\mathbf{1}_A)=\varphi(x)f(\mathbf{1}_A)$ for every $x\in I$ and every $A\subseteq [n]$.

Clearly, every weakly homogeneous function $f$ satisfies $f(\mathbf{0})=0$ (take $x=0$ in the definition).

The following proposition provides necessary and sufficient conditions on a nonconstant quasi-Lov\'asz extension $f\colon I^n\to\R$ for the
function $f_0$ to be weakly homogeneous.

\begin{proposition}\label{lemma:ds8686}
Assume $[0,1]\subseteq I\subseteq\R_+$. Let $f\colon I^n\to\R$ be a nonconstant quasi-Lov\'asz extension, $f=L\circ\varphi$. Then the following
conditions are equivalent.
\begin{enumerate}
\item[$(i)$] $f_0$ is weakly homogeneous.

\item[$(ii)$] There exists $A\subseteq [n]$ such that $f_0(\mathbf{1}_A)\neq 0$.

\item[$(iii)$] $\varphi(1)\neq 0$.
\end{enumerate}
In this case we have $f_0(x\mathbf{1}_A)=\frac{\varphi(x)}{\varphi(1)}\, f_0(\mathbf{1}_A)$ for every $x\in I$ and every $A\subseteq [n]$.
\end{proposition}

\begin{proof}
Let us prove that $(i)\Rightarrow (ii)$ by contradiction. Assume that $f_0(\mathbf{1}_A)= 0$ for every $A\subseteq [n]$. Since $f_0$ is weakly
homogeneous, we must have $f_0(x\mathbf{1}_A)=0$ for every $x\in I$ and every $A\subseteq [n]$. By (\ref{eq:dsa786}), we then have
$\varphi\equiv 0$ or $L_0(\mathbf{1}_A)=0$ for every $A\subseteq [n]$. In either case, by (\ref{eq:sdfaf678dew}), we have $f_0\equiv 0$, i.e.\
$f$ is constant, a contradiction.

Let us prove that $(ii)\Rightarrow (iii)$ by contradiction. If we had $\varphi(1)=0$, then by (\ref{eq:dsa786}) we would have
$f_0(\mathbf{1}_A)= 0$ for every $A\subseteq [n]$, a contradiction.

Let us prove that $(iii)\Rightarrow (i)$. By (\ref{eq:dsa786}), we have $f_0(x\mathbf{1}_A)=\frac{\varphi(x)}{\varphi(1)}\, f_0(\mathbf{1}_A)$,
which shows that $f_0$ is weakly homogeneous.
\end{proof}

\begin{remark}\label{rem:as9f67}
\begin{enumerate}
\item[$(a)$] If $[0,1]\varsubsetneq I\subseteq\R_+$, then the quasi-Lov\'asz extension $f\colon I^n\to\R$ defined by $f(\bfx)=\bigwedge_{i\in
[n]}\varphi(x_i)$, where $\varphi(x)=0\vee(x-1)$, is not weakly homogeneous.

\item[$(b)$] When $I=[0,1]$, the assumption that $f$ is nonconstant implies immediately that $\varphi(1)\neq 0$. We then see by
Proposition~\ref{lemma:ds8686} that $f_0$ is weakly homogeneous. Note also that, if $f$ is constant, then $f_0\equiv 0$ is clearly weakly
homogeneous. Thus, for any quasi-Lov\'asz extension $f\colon [0,1]^n\to\R$, the function $f_0$ is weakly homogeneous.
\end{enumerate}
\end{remark}

Dually, we say that a function $f\colon I^n\to\R$, where $I\subseteq\R_-$, is \emph{weakly homogeneous} if there exists a nondecreasing function
$\varphi\colon I\to\R$ satisfying $\varphi(0)=0$ such that $f(x\mathbf{1}_A)=-\varphi(x)f(-\mathbf{1}_A)$ for every $x\in I$ and every
$A\subseteq [n]$.

Using (\ref{eq:sdfaf678dew2}) (instead of (\ref{eq:sdfaf678dew})), we can easily obtain the following negative counterparts of
Lemma~\ref{lemma:sdf7d} and Proposition~\ref{lemma:ds8686}.

\begin{lemma}\label{lemma:sdf7daa}
Assume $I\subseteq\R_-$. For every quasi-Lov\'asz extension $f\colon I^n\to\R$, $f=L\circ\varphi$, we have
$$
f_0(x\mathbf{1}_A)=-\varphi(x)L_0(-\mathbf{1}_A),\qquad x\in I,~A\subseteq [n].
$$
\end{lemma}

\begin{proposition}\label{lemma:ds8686aa}
Assume $[-1,0]\subseteq I\subseteq\R_-$. Let $f\colon I^n\to\R$ be a nonconstant quasi-Lov\'asz extension, $f=L\circ\varphi$. Then the following
conditions are equivalent.
\begin{enumerate}
\item[$(i)$] $f_0$ is weakly homogeneous.

\item[$(ii)$] There exists $A\subseteq [n]$ such that $f_0(-\mathbf{1}_A)\neq 0$.

\item[$(iii)$] $\varphi(-1)\neq 0$.
\end{enumerate}
In this case we have $f_0(x\mathbf{1}_A)=\frac{\varphi(x)}{\varphi(-1)}\, f_0(-\mathbf{1}_A)$ for every $x\in I$ and every $A\subseteq [n]$.
\end{proposition}

Assume now that $-x\in I$ whenever $x\in I$, that is, $I$ is centered at $0$. A \emph{symmetric quasi-Lov\'asz extension} is a function $f\colon
I^n\to\R$ defined by
$$
f=\check{L}\circ\varphi,
$$
where $\check{L}\colon\R^n\to\R$ is a symmetric Lov\'asz extension and $\varphi\colon I\to\R$ is a nondecreasing odd function.

Combining Lemmas~\ref{lemma:sdf7d} and \ref{lemma:sdf7daa} with the fact that $\check{L}_0$ and $\varphi$ are odd functions, we obtain
immediately the following result.

\begin{lemma}\label{lemma:asd68}
Assume that $I$ is centered at $0$. For every symmetric quasi-Lov\'asz extension $f\colon I^n\to\R$, $f=\check{L}\circ\varphi$, we have
\begin{equation}\label{eq:dsa786aabv}
f_0(x\mathbf{1}_A) = \varphi(x)\check{L}_0(\mathbf{1}_A),\qquad x\in I,~A\subseteq [n].
\end{equation}
\end{lemma}

We say that a function $f\colon I^n\to\R$, where $I$ centered at $0$, is \emph{oddly homogeneous} if there exists a nondecreasing odd function
$\varphi\colon I\to\R$ such that $f(x\mathbf{1}_A)=\varphi(x)f(\mathbf{1}_A)$ for every $x\in I$ and every $A\subseteq [n]$.

Clearly, for every oddly homogeneous function $f$, the functions $f|_{I_+^n}$ and $f|_{I_-^n}$ are weakly homogeneous.

The following proposition provides necessary and sufficient conditions on a nonconstant symmetric quasi-Lov\'asz extension $f\colon I^n\to\R$
for the function $f_0$ to be oddly homogeneous.

\begin{proposition}\label{lemma:ds8686sym}
Assume that $I$ is centered at $0$ with $[-1,1]\subseteq I$. Let $f\colon I^n\to\R$ be a symmetric quasi-Lov\'asz extension,
$f=\check{L}\circ\varphi$, such that $f|_{I_+^n}$ or $f|_{I_-^n}$ is nonconstant. Then the following conditions are equivalent.
\begin{enumerate}
\item[$(i)$] $f_0$ is oddly homogeneous.

\item[$(ii)$] There exists $A\subseteq [n]$ such that $f_0(\mathbf{1}_A)\neq 0$.

\item[$(iii)$] $\varphi(1)\neq 0$.
\end{enumerate}
In this case we have $f_0(x\mathbf{1}_A)=\frac{\varphi(x)}{\varphi(1)}\, f_0(\mathbf{1}_A)$ for every $x\in I$ and every $A\subseteq [n]$.
\end{proposition}

\begin{proof}
Since $f|_{I_+^n}$ or $f|_{I_-^n}$ is nonconstant and $f_0$ is odd, we have $f_0|_{I_+^n}\not\equiv 0$ and $f|_{I_-^n}\not\equiv 0$.

The implications $(i)\Rightarrow (ii)\Rightarrow (iii)$ follow from Proposition~\ref{lemma:ds8686} and the fact that, if $f_0$ is oddly
homogeneous, then $f_0|_{I_+^n}$ is weakly homogeneous.

Now, assume that $(iii)$ holds. Since $\check{L}_0$ and $\varphi$ are odd, by Propositions~\ref{lemma:ds8686} and \ref{lemma:ds8686aa} we
clearly have $f_0(x\mathbf{1}_A)=\frac{\varphi(x)}{\varphi(1)}\, f_0(\mathbf{1}_A)$ for every $x\in I$ and every $A\subseteq [n]$, which shows
that $(i)$ also holds.
\end{proof}

\begin{remark}
Similarly to Remark~\ref{rem:as9f67}$(b)$, we see that, for any symmetric quasi-Lov\'asz extension $f\colon [-1,1]^n\to\R$, the function $f_0$
is oddly homogeneous.
\end{remark}

\section{Comonotonic modularity}

Recall that a function $f\colon I^n\to\R$ is said to be \emph{modular} (or a \emph{valuation}) if
\begin{equation}\label{eq:sdf75}
f(\bfx)+f(\bfx')=f(\bfx\wedge\bfx')+f(\bfx\vee\bfx')
\end{equation}
for every $\bfx,\bfx'\in I^n$. It was proved (see Topkis \cite[Thm~3.3]{Top78}) that a function $f\colon I^n\to\R$ is modular if and only if it
is \emph{separable}, that is, there exist $n$ functions $f_i\colon I\to\R$, $i\in [n]$, such that $f=\sum_{i\in [n]}f_i$.\footnote{This result
still holds in the more general framework where $f$ is defined on a product of chains.} In particular, any $1$-place function $f\colon I\to\R$
is modular.

Two $n$-tuples $\bfx,\bfx'\in I^n$ are said to be \emph{comonotonic} if there exists $\sigma\in S_n$ such that $\bfx,\bfx'\in I^n_{\sigma}$. A
function $f\colon I^n\to\R$ is said to be \emph{comonotonically modular} (or a \emph{comonotonic valuation}) if (\ref{eq:sdf75}) holds for every
comonotonic $n$-tuples $\bfx,\bfx'\in I^n$. This notion was considered in the special case when $I=[0,1]$ in \cite{MesMes11}. We observe that,
for any function $f\colon I^n\to\R$, condition (\ref{eq:sdf75}) holds for every $\bfx,\bfx'\in I^n$ of the forms $\bfx=x\mathbf{1}_A$ and
$\bfx'=x'\mathbf{1}_A$, where $x,x'\in I$ and $A\subseteq [n]$.

Observe also that, for every $\bfx\in\R^n_+$ and every $c\in\R_+$, we have
$$
\bfx-\bfx\wedge c=[\bfx]_c-[\bfx]_c\wedge c.
$$
This motivates the following definition. We say that a function $f\colon I^n\to\R$, where $I\subseteq\R_+$, is \emph{invariant under horizontal
min-differences} if, for every $\bfx\in I^n$ and every $c\in I$, we have
\begin{equation}\label{eq:HminDif}
f(\bfx)-f(\bfx\wedge c)=f([\bfx]_c)-f([\bfx]_c\wedge c).
\end{equation}
Dually, we say that a function $f\colon I^n\to\R$, where $I\subseteq\R_-$, is \emph{invariant under horizontal max-differences} if, for every
$\bfx\in I^n$ and every $c\in I$, we have
\begin{equation}\label{eq:HmaxDif}
f(\bfx)-f(\bfx\vee c)=f([\bfx]^c)-f([\bfx]^c\vee c).
\end{equation}

\begin{fact}\label{fact:min-max}
Assume $I\subseteq\R_+$. A function $f\colon (-I)^n\to\R$, where $-I=\{-x:x\in I\}$, is invariant under horizontal max-differences if and only
if the function $f'\colon I^n\to\R$, defined by $f'(\bfx)=f(-\bfx)$ for every $\bfx\in I^n$, is invariant under horizontal min-differences.
\end{fact}

We observe that, for any function $f\colon I^n\to\R$, where $I\subseteq\R_+$, condition (\ref{eq:HminDif}) holds for every $\mathbf{x}\in I^n$
of the form $\mathbf{x}=x\mathbf{1}_A$, where $x\in I$ and $A\subseteq [n]$. Dually, for any function $f\colon I^n\to\R$, where
$I\subseteq\R_-$, condition (\ref{eq:HmaxDif}) holds for every tuple $\mathbf{x}\in I^n$ of the form $\mathbf{x}=x\mathbf{1}_A$, where $x\in I$
and $A\subseteq [n]$.

We also observe that a function $f$ is comonotonically modular (resp.\ invariant under horizontal min-differences, invariant under horizontal
max-differences) if and only if so is the function $f_0$.

\begin{theorem}\label{thm:horizontaldiff}
Assume $I\subseteq\R_+$ and let $f\colon I^n\to\R$ be a function. Then the following assertions are equivalent.
\begin{enumerate}
\item[$(i)$] $f$ is comonotonically modular.

\item[$(ii)$] $f$ is invariant under horizontal min-differences.

\item[$(iii)$] There exists a function $g\colon I^n\to\R$ such that, for every $\sigma\in S_n$ and every $\mathbf{x}\in I^n_\sigma$, we have
\begin{equation}\label{eq:H2a}
f(\mathbf{x})=g(\mathbf{0})+\sum_{i\in [n]}\big
(g(x_{\sigma(i)}\mathbf{1}_{A_{\sigma}^{\uparrow}(i)})-g(x_{\sigma(i)}\mathbf{1}_{A_{\sigma}^{\uparrow}(i+1)})\big).
\end{equation}
In this case, we can choose $g=f$.
\end{enumerate}
\end{theorem}

\begin{proof}
$(i)\Rightarrow (iii)$ Let $\sigma\in S_n$ and $\bfx\in I^n_{\sigma}$. By comonotonic modularity, for every $i\in [n-1]$ we have
$$
f(x_{\sigma(i)}\mathbf{1}_{A_{\sigma}^{\uparrow}(i)})+f(\bfx^0_{A_{\sigma}^{\downarrow}(i)})=f(x_{\sigma(i)}\mathbf{1}_{A_{\sigma}^{\uparrow}(i+1)})
+f(\bfx^0_{A_{\sigma}^{\downarrow}(i-1)}),
$$
that is,
\begin{equation}\label{eq:sa8f76}
f(\bfx^0_{A_{\sigma}^{\downarrow}(i-1)}) =
\big(f(x_{\sigma(i)}\mathbf{1}_{A_{\sigma}^{\uparrow}(i)})-f(x_{\sigma(i)}\mathbf{1}_{A_{\sigma}^{\uparrow}(i+1)})\big)+f(\bfx^0_{A_{\sigma}^{\downarrow}(i)}).
\end{equation}
By using (\ref{eq:sa8f76}) for $i=1,\ldots,n-1$, we obtain (\ref{eq:H2a}) with $g=f$.

$(iii)\Rightarrow (i)$ For every $\sigma\in S_n$ and every $\bfx,\bfx'\in I^n_\sigma$, we have
\begin{eqnarray*}
f_0(\bfx)+f_0(\bfx') &=& \sum_{i\in [n]}\big(g(x_{\sigma(i)}\mathbf{1}_{A_{\sigma}^{\uparrow}(i)})+g(x'_{\sigma(i)}\mathbf{1}_{A_{\sigma}^{\uparrow}(i)})\big)\\
&& \null -\sum_{i\in
[n]}\big(g(x_{\sigma(i)}\mathbf{1}_{A_{\sigma}^{\uparrow}(i+1)})+g(x'_{\sigma(i)}\mathbf{1}_{A_{\sigma}^{\uparrow}(i+1)})\big)
\end{eqnarray*}
and, since $g$ satisfies property (\ref{eq:sdf75}) for every $\bfx,\bfx'\in I^n$ of the forms $\bfx=x\mathbf{1}_A$ and $\bfx'=x'\mathbf{1}_A$,
where $x,x'\in I$ and $A\subseteq [n]$, we have that $(i)$ holds.

$(ii)\Rightarrow (iii)$ Let $\sigma\in S_n$ and $\bfx\in I^n_{\sigma}$. There exists $p\in [n]$ such that $x_{\sigma(1)}=\cdots =x_{\sigma(p)}<
x_{\sigma(p+1)}$.\footnote{Here $x_{\sigma(n+1)}=+\infty$.} Then, using (\ref{eq:HminDif}) with $c=x_{\sigma(1)}$, we get
$$
f(\mathbf{x})-f(x_{\sigma(1)} \mathbf{1}_{A_{\sigma}^{\uparrow}(1)}) = f(\mathbf{x}_{A_{\sigma}^{\downarrow}(p)}^0)-f(x_{\sigma(1)}
\mathbf{1}_{A_{\sigma}^{\uparrow}(p+1)}).
$$
Using a telescoping sum and the fact that $x_{\sigma(1)}=\cdots =x_{\sigma(p)}$, we obtain
\begin{eqnarray}
f(\mathbf{x}) &=& \big(f(x_{\sigma(1)} \mathbf{1}_{A_{\sigma}^{\uparrow}(1)}) - f(x_{\sigma(1)}
\mathbf{1}_{A_{\sigma}^{\uparrow}(p+1)})\big)+f(\mathbf{x}_{A_{\sigma}^{\downarrow}(p)}^0)\nonumber\\
&=& \sum_{i=1}^p \big(f(x_{\sigma(i)} \mathbf{1}_{A_{\sigma}^{\uparrow}(i)}) - f(x_{\sigma(i)}
\mathbf{1}_{A_{\sigma}^{\uparrow}(i+1)})\big)+f(\mathbf{x}_{A_{\sigma}^{\downarrow}(p)}^0).\label{eq:sdf67}
\end{eqnarray}
If $p=n-1$ or $p=n$, then (\ref{eq:H2a}) holds with $g=f$. Otherwise, there exists $q\in [n-p]$ such that $x_{\sigma(p+1)}=\cdots
=x_{\sigma(p+q)}< x_{\sigma(p+q+1)}$ and we expand the last term in (\ref{eq:sdf67}) similarly by using (\ref{eq:HminDif}) with
$c=x_{\sigma(p+1)}$. We then repeat this procedure until the last term is $f(\mathbf{0})$, thus obtaining (\ref{eq:H2a}) with $g=f$.

To illustrate, suppose $x_1<x_2=x_3<x_4$. Then
$$
f(x_1,x_2,x_3,x_4) = \big(f(x_1,x_1,x_1,x_1)-f(0,x_1,x_1,x_1)\big)+f(0,x_2,x_3,x_4)
$$
with
\begin{eqnarray*}
f(0,x_2,x_3,x_4) 
&=& \big(f(0,x_2,x_2,x_2)-f(0,0,x_2,x_2)\big)\\
&& \null +\big(f(0,0,x_3,x_3)-f(0,0,0,x_3)\big) +f(0,0,0,x_4)
\end{eqnarray*}
and
$$
f(0,0,0,x_4) = \big(f(0,0,0,x_4)-f(0,0,0,0)\big)+f(0,0,0,0).
$$

$(iii)\Rightarrow (ii)$ For every $\sigma\in S_n$, every $\mathbf{x}\in I^n_\sigma$, and every $c\in I$, we have
\begin{eqnarray*}
f(\bfx)-f(\bfx\wedge c) &=& \sum_{i\in [n]}\big(g(x_{\sigma(i)}\mathbf{1}_{A_{\sigma}^{\uparrow}(i)})-g((x_{\sigma(i)}\wedge
c)\mathbf{1}_{A_{\sigma}^{\uparrow}(i)})\big)\\
&& \null -\sum_{i\in [n]}\big(g(x_{\sigma(i)}\mathbf{1}_{A_{\sigma}^{\uparrow}(i+1)})-g((x_{\sigma(i)}\wedge
c)\mathbf{1}_{A_{\sigma}^{\uparrow}(i+1)})\big)
\end{eqnarray*}
and, since $g$ satisfies property (\ref{eq:HminDif}) for every $\mathbf{x}\in I^n$ of the form $\mathbf{x}=x\mathbf{1}_A$, where $x\in I$ and
$A\subseteq [n]$, we have that $(ii)$ holds.
\end{proof}

\begin{remark}
The equivalence between $(i)$ and $(iii)$ in Theorem~\ref{thm:horizontaldiff} generalizes Theorem~1 in \cite{MesMes11}, which describes the
class of comonotonically modular functions $f\colon [0,1]^n\to [0,1]$ under the additional conditions of symmetry and idempotence.
\end{remark}

The following theorem is the negative counterpart of Theorem \ref{thm:horizontaldiff} and its proof follows dually by taking into account
Fact~\ref{fact:min-max}.

\begin{theorem}\label{thm:horizontalmaxdiff}
Assume $I\subseteq\R_-$ and let $f\colon I^n\to\R$ be a function. Then the following assertions are equivalent.
\begin{enumerate}
\item[$(i)$] $f$ is comonotonically modular.

\item[$(ii)$] $f$ is invariant under horizontal max-differences.

\item[$(iii)$] There exists a function $g\colon I^n\to\R$ such that, for every $\sigma\in S_n$ and every $\mathbf{x}\in I^n_\sigma$, we have
$$
f(\mathbf{x})=g(\mathbf{0})+\sum_{i\in [n]}\big
(g(x_{\sigma(i)}\mathbf{1}_{A_{\sigma}^{\downarrow}(i)})-g(x_{\sigma(i)}\mathbf{1}_{A_{\sigma}^{\downarrow}(i-1)})\big).
$$
In this case, we can choose $g=f$.
\end{enumerate}
\end{theorem}

We observe that if $f\colon I^n\to\R$ is comonotonically modular then necessarily
\begin{equation}\label{eq:sdf75dd}
f_0(\bfx)=f_0(\bfx^+)+f_0(-\bfx^-)
\end{equation}
(take $\bfx'=\mathbf{0}$ in (\ref{eq:sdf75})).
%

We may now characterize the class of comonotonically modular functions on an arbitrary interval $I$ containing $0$.

\begin{theorem}\label{thm:wqer87we}
For any function $f\colon I^n\to\R$, the following assertions are equivalent.
\begin{enumerate}
\item[$(i)$] $f$ is comonotonically modular.

\item[$(ii)$] There exist $g\colon I_+^n\to\R$ comonotonically modular (or invariant under horizontal min-differences) and $h\colon I_-^n\to\R$
comonotonically modular (or invariant under horizontal max-differences) such that $f_0(\bfx)=g_0(\bfx^+)+h_0(-\bfx^-)$ for every $\bfx\in I^n$.
In this case, we can choose $g=f|_{I_+^n}$ and $h=f|_{I_-^n}$.

\item[$(iii)$] There exist $g\colon I_+^n\to\R$ and $h\colon I_-^n\to\R$ such that, for every $\sigma\in S_n$ and every $\bfx\in I^n_{\sigma}$,
\begin{eqnarray*}
f_0(\bfx) &=& \sum_{1\leqslant i\leqslant p}\big
(h(x_{\sigma(i)}\mathbf{1}_{A_{\sigma}^{\downarrow}(i)})-h(x_{\sigma(i)}\mathbf{1}_{A_{\sigma}^{\downarrow}(i-1)})\big)\\
&& \null + \sum_{p+1\leqslant i\leqslant n}\big
(g(x_{\sigma(i)}\mathbf{1}_{A_{\sigma}^{\uparrow}(i)})-g(x_{\sigma(i)}\mathbf{1}_{A_{\sigma}^{\uparrow}(i+1)})\big),
\end{eqnarray*}
where $p\in\{0,\ldots,n\}$ is such that $x_{\sigma(p)}<0\leqslant x_{\sigma(p+1)}$. In this case, we can choose $g=f|_{I_+^n}$ and
$h=f|_{I_-^n}$.
\end{enumerate}
\end{theorem}

\begin{proof}
$(i)\Rightarrow (ii)$ Follows from (\ref{eq:sdf75dd}) and Theorems~\ref{thm:horizontaldiff} and \ref{thm:horizontalmaxdiff}.

$(ii)\Rightarrow (iii)$ Follows from Theorems~\ref{thm:horizontaldiff} and \ref{thm:horizontalmaxdiff}.

$(iii)\Rightarrow (i)$ Clearly, $f_0$ satisfies (\ref{eq:sdf75dd}). Let $\sigma\in S_n$ and $\bfx,\bfx'\in I^n_{\sigma}$. By (\ref{eq:sdf75dd})
we have
$$
f_0(\bfx)+f_0(\bfx')=f_0(\bfx^+)+f_0(\bfx'^+)+f_0(-\bfx^-)+f(-\bfx'^-)
$$
Using Theorems~\ref{thm:horizontaldiff} and \ref{thm:horizontalmaxdiff}, we see that this identity can be rewritten as
\begin{eqnarray*}
f_0(\bfx)+f_0(\bfx') &=& f_0(\bfx^+\wedge\bfx'^+)+f_0(\bfx^+\vee\bfx'^+)+f_0(-\bfx^-\wedge -\bfx'^-)+f_0(-\bfx^-\vee -\bfx'^-)\\
&=& f_0\big((\bfx\wedge\bfx')^+\big)+f_0\big((\bfx\vee\bfx')^+\big)+f_0\big(-(\bfx\wedge\bfx')^-\big)+f_0\big(-(\bfx\vee\bfx')^-\big),
\end{eqnarray*}
which, by (\ref{eq:sdf75dd}), becomes $f_0(\bfx)+f_0(\bfx')=f_0(\bfx\wedge\bfx')+f_0(\bfx\vee\bfx')$. Therefore, $f_0$ is comonotonically
modular and, hence, so is $f$.
\end{proof}

From Theorem~\ref{thm:wqer87we} we obtain the ``comonotonic'' analogue of Topkis' characterization \cite{Top78} of modular functions as
separable functions, and which provides an alternative description of comonotonically modular functions. We make use of the following fact.

\begin{fact}\label{fact:f78}
Let $J$ be any nonempty real interval, possibly unbounded, and let $c\in J$. A function $g\colon J^n\to\R$ is modular (resp.\ comonotonically
modular) if and only if the function $f\colon I^n\to\R$, defined by $f(\bfx)=g(\bfx+c\mathbf{1})$, where $I=J-c=\{z-c:z\in J\}$, is  modular
(resp.\ comonotonically modular).
\end{fact}

\begin{corollary}\label{cor:8sa67}
Let $J$ be any nonempty real interval, possibly unbounded. A function $f\colon J^n\to\R$ is comonotonically modular if and only if it is
comonotonically separable, that is, for every $\sigma\in S_n$, there exist functions $f^{\sigma}_i\colon J\to\R$, $i\in [n]$, such that
$$
f(\bfx)=\sum_{i=1}^nf^{\sigma}_i(x_{\sigma(i)})=\sum_{i=1}^nf^{\sigma}_{\sigma^{-1}(i)}(x_i),\qquad \bfx\in J^n\cap\R^n_{\sigma}.
$$
\end{corollary}

\begin{proof}
(Necessity) By Fact~\ref{fact:f78} we can assume that $J$ contains the origin. The result then follows from the equivalence $(i)\Leftrightarrow
(iii)$ stated in Theorem~\ref{thm:wqer87we}.

(Sufficiency) For every $\sigma\in S_n$ and every $i\in [n]$, the function $f^{\sigma}_{\sigma^{-1}(i)}$ is clearly modular and hence
comonotonically modular. Since the class of comonotonically modular functions is closed under addition, the proof is now complete.
\end{proof}

\section{Axiomatization and representation of quasi-Lov\'asz extensions}

We now provide axiomatizations of the class of quasi-Lov\'asz extensions and describe all possible factorizations of quasi-Lov\'asz extensions
into compositions of Lov\'asz extensions with $1$-place nondecreasing functions.

\begin{theorem}\label{thm:CanonicalQuasi-Lovasz}
Assume $[0,1]\subseteq I\subseteq\R_+$ and let $f\colon I^n\to\R$ be a nonconstant function. Then the following assertions are equivalent.
\begin{enumerate}
\item[$(i)$] $f$ is a quasi-Lov\'asz extension and there exists $A\subseteq [n]$ such that $f_0(\mathbf{1}_A)\neq 0$.

\item[$(ii)$] $f$ is comonotonically modular (or invariant under horizontal min-differences) and $f_0$ is weakly homogeneous.

\item[$(iii)$] There is a nondecreasing function $\varphi_f\colon I\to\R$ satisfying $\varphi_f(0)=0$ and $\varphi_f(1)=1$ such that
$f=L_{f|_{\B^n}}\circ\varphi_f$.
\end{enumerate}
\end{theorem}

\begin{proof}
Let us prove that $(i)\Rightarrow (ii)$. By definition, we have $f=L\circ\varphi$, where $L\colon\R^n\to\R$ is a Lov\'asz extension and
$\varphi\colon I\to\R$ is a nondecreasing function satisfying $\varphi(0)=0$. By Proposition~\ref{lemma:ds8686}, $f_0$ is weakly homogeneous.
Moreover, by (\ref{eq:sdfaf678dew}) and (\ref{eq:dsa786}) we have that, for every $\sigma\in S_n$ and every $\bfx\in I^n_{\sigma}$,
\begin{eqnarray*}
f(\bfx) &=&
f(\mathbf{0})+\sum_{i\in [n]}\varphi(x_{\sigma(i)})\,\big(L_0(\mathbf{1}_{A_{\sigma}^{\uparrow}(i)})-L_0(\mathbf{1}_{A_{\sigma}^{\uparrow}(i+1)})\big)\\
&=& f(\mathbf{0})+\sum_{i\in [n]}\big
(f(x_{\sigma(i)}\mathbf{1}_{A_{\sigma}^{\uparrow}(i)})-f(x_{\sigma(i)}\mathbf{1}_{A_{\sigma}^{\uparrow}(i+1)})\big).
\end{eqnarray*}
Theorem~\ref{thm:horizontaldiff} then shows that $f$ is comonotonically modular.

Let us prove that $(ii)\Rightarrow (iii)$. Since $f$ is comonotonically modular, by Theorem~\ref{thm:horizontaldiff} it follows that, for every
$\sigma\in S_n$ and every $\bfx\in I^n_{\sigma}$,
$$
f(\mathbf{x})=f(\mathbf{0})+\sum_{i\in [n]}\big
(f(x_{\sigma(i)}\mathbf{1}_{A_{\sigma}^{\uparrow}(i)})-f(x_{\sigma(i)}\mathbf{1}_{A_{\sigma}^{\uparrow}(i+1)})\big),
$$
and, since $f_0$ is weakly homogeneous,
\begin{equation}\label{eq:sa9df78}
f(\mathbf{x})=f(\mathbf{0})+\sum_{i\in [n]}\varphi_f(x_{\sigma(i)})\,\big
(f(\mathbf{1}_{A_{\sigma}^{\uparrow}(i)})-f(\mathbf{1}_{A_{\sigma}^{\uparrow}(i+1)})\big)
\end{equation}
for some nondecreasing function $\varphi_f\colon I\to\R$ satisfying $\varphi_f(0)=0$. By (\ref{eq:sdfaf678dew}), we then obtain
$f=L_{f|_{\B^n}}\circ\varphi_f$. Finally, by (\ref{eq:sa9df78}) we have that, for every $A\subseteq [n]$,
$$
f_0(\mathbf{1}_A)=\varphi_f(1)f_0(\mathbf{1}_A).
$$
Since there exists $A\subseteq [n]$ such that $f_0(\mathbf{1}_A)\neq 0$ (for otherwise, we would have $f_0\equiv 0$ by (\ref{eq:sa9df78})), we
obtain $\varphi_f(1)=1$.

The implication $(iii)\Rightarrow (i)$ follows from Proposition~\ref{lemma:ds8686}.
\end{proof}

Let $f\colon I^n\to\R$ be a quasi-Lov\'asz extension, where $[0,1]\subseteq I\subseteq\R_+$, for which there exists $A^*\subseteq [n]$ such that
$f_0(\mathbf{1}_{A^*})\neq 0$. Then the inner function $\varphi_f$ introduced in Theorem~\ref{thm:CanonicalQuasi-Lovasz} is unique. Indeed, by
Proposition~\ref{lemma:ds8686}, we have $f_0(x\mathbf{1}_A)=\varphi_f(x)f_0(\mathbf{1}_A)$ for every $x\in I$ and every $A\subseteq [n]$. The
function $\varphi_f$ is then defined by
$$
\varphi_f(x)=\frac{f_0(x\mathbf{1}_{A^*})}{f_0(\mathbf{1}_{A^*})}~,\qquad x\in I.
$$

We can now describe the possible factorizations of $f$ into compositions of Lov\'asz extensions with nondecreasing functions.

\begin{theorem}\label{thm:6sad8f65}
Assume $[0,1]\subseteq I\subseteq\R_+$ and let $f\colon\ I^n\to\R$ be a quasi-Lov\'asz extension, $f=L\circ\varphi$. Then there exists
$A^*\subseteq [n]$ such that $f_0(\mathbf{1}_{A^*})\neq 0$ if and only if there exists $a>0$ such that $\varphi=a\,\varphi_f$ and $L_0=\frac
1a(L_{f|_{\B^n}})_0$.
\end{theorem}

\begin{proof}
(Sufficiency) We have $f_0=L_0\circ\varphi=(L_{f|_{\B^n}})_0\circ\varphi_f$, and by Theorem~\ref{thm:CanonicalQuasi-Lovasz} we see that the
conditions are sufficient.

(Necessity) By Proposition~\ref{lemma:ds8686}, we have
$$
\frac{\varphi(x)}{\varphi(1)}=\frac{f_0(x\mathbf{1}_{A^*})}{f_0(\mathbf{1}_{A^*})}=\varphi_f(x).
$$
We then have $\varphi=a\,\varphi_f$ for some $a>0$. Moreover, for every $\bfx\in\B^n$, we have
\begin{eqnarray*}
(L_{f|_{\B^n}})_0(\bfx) &=& \big((L_{f|_{\B^n}})_0\circ\varphi_f\big)(\bfx)%
~=~ f_0(\bfx)%
~=~ (L_0\circ\varphi)(\bfx)\\
&=& a(L_0\circ\varphi_f)(\bfx)%
~=~ a\,L_0(\bfx).
\end{eqnarray*}
Since a Lov\'asz extension is uniquely determined by its values on $\B^n$, we have $(L_{f|_{\B^n}})_0=a\, L_0$.
\end{proof}

The following two theorems are the negative counterparts of Theorems~\ref{thm:CanonicalQuasi-Lovasz} and \ref{thm:6sad8f65} and their proofs
follow dually.

\begin{theorem}\label{thm:CanonicalQuasi-Lovaszaa}
Assume $[-1,0]\subseteq I\subseteq\R_-$ and let $f\colon I^n\to\R$ be a nonconstant function. Then the following assertions are equivalent.
\begin{enumerate}
\item[$(i)$] $f$ is a quasi-Lov\'asz extension and there exists $A\subseteq [n]$ such that $f_0(-\mathbf{1}_A)\neq 0$.

\item[$(ii)$] $f$ is comonotonically modular (or invariant under horizontal max-differences) and $f_0$ is weakly homogeneous.

\item[$(iii)$] There is a nondecreasing function $\varphi_f\colon I\to\R$ satisfying $\varphi_f(0)=0$ and $\varphi_f(-1)=-1$ such that
$f=L_{f|_{-\B^n}}\circ\varphi_f$.
\end{enumerate}
\end{theorem}

\begin{theorem}\label{thm:dsaf5s}
Assume $[-1,0]\subseteq I\subseteq\R_-$ and let $f\colon I^n\to\R$ be a quasi-Lov\'asz extension, $f=L\circ\varphi$. Then there exists
$A^*\subseteq [n]$ such that $f_0(-\mathbf{1}_{A^*})\neq 0$ if and only if there exists $a>0$ such that $\varphi=a\,\varphi_f$ and $L_0=\frac
1a(L_{f|_{-\B^n}})_0$.
\end{theorem}

\begin{remark}
If $I=[0,1]$ (resp.\ $I=[-1,0]$), then the ``nonconstant'' assumption and the second condition in assertion $(i)$ of
Theorem~\ref{thm:CanonicalQuasi-Lovasz} (resp.\ Theorem~\ref{thm:CanonicalQuasi-Lovaszaa}) can be dropped off.
\end{remark}

\section{Axiomatization and representation of symmetric quasi-Lov\'asz extensions}

We now provide an axiomatization of the class of symmetric quasi-Lov\'asz extensions and describe all possible factorizations of symmetric
quasi-Lov\'asz extensions into compositions of symmetric Lov\'asz extensions with $1$-place nondecreasing odd functions. We proceed in complete
analogy as in the previous section.

\begin{theorem}\label{thm:CanonicalSymQuasi-Lovasz}
Assume that $I$ is centered at $0$ with $[-1,1]\subseteq I$ and let $f\colon I^n\to\R$ be a function such that $f|_{I_+^n}$ or $f|_{I_-^n}$ is
nonconstant. Then the following assertions are equivalent.
\begin{enumerate}
\item[$(i)$] $f$ is a symmetric quasi-Lov\'asz extension and there exists $A\subseteq [n]$ such that $f_0(\mathbf{1}_A)\neq 0$.

\item[$(ii)$] $f$ is comonotonically modular and $f_0$ is oddly homogeneous.

\item[$(iii)$] There is a nondecreasing odd function $\varphi_f\colon I\to\R$ satisfying $\varphi_f(1)=1$ such that
$f=\check{L}_{f|_{\B^n}}\circ\varphi_f$.
\end{enumerate}
\end{theorem}

\begin{proof}
Let us prove that $(i)\Rightarrow (ii)$. By definition, we have $f=\check{L}\circ\varphi$, where $\check{L}\colon\R^n\to\R$ is a symmetric
Lov\'asz extension and $\varphi\colon I\to\R$ is a nondecreasing odd function. By Proposition~\ref{lemma:ds8686sym}, $f_0$ is oddly homogeneous.
Moreover, for every $\sigma\in S_n$ and every $\bfx\in I^n_{\sigma}$, by (\ref{eq:sdfsfd65dsf}) and (\ref{eq:dsa786aabv}) we have
\begin{eqnarray*}
f(\bfx) &=& f(\mathbf{0})+\sum_{1\leqslant i\leqslant p}
\varphi(x_{\sigma(i)})\,\big(L_0(\mathbf{1}_{A_{\sigma}^{\downarrow}(i)})-L_0(\mathbf{1}_{A_{\sigma}^{\downarrow}(i-1)})\big)\\
&& \null +\sum_{p+1\leqslant i\leqslant n}
\varphi(x_{\sigma(i)})\,\big(L_0(\mathbf{1}_{A_{\sigma}^{\uparrow}(i)})-L_0(\mathbf{1}_{A_{\sigma}^{\uparrow}(i+1)})\big)\\
&=& f(\mathbf{0})+\sum_{1\leqslant i\leqslant p} \big
(f(x_{\sigma(i)}\mathbf{1}_{A_{\sigma}^{\downarrow}(i)})-f(x_{\sigma(i)}\mathbf{1}_{A_{\sigma}^{\downarrow}(i-1)})\big)\\
&& \null +\sum_{p+1\leqslant i\leqslant n} \big
(f(x_{\sigma(i)}\mathbf{1}_{A_{\sigma}^{\uparrow}(i)})-f(x_{\sigma(i)}\mathbf{1}_{A_{\sigma}^{\uparrow}(i+1)})\big),
\end{eqnarray*}
where $p\in\{0,\ldots,n\}$ is such that $x_{\sigma(p)}<0\leqslant x_{\sigma(p+1)}$. By Theorem~\ref{thm:wqer87we} it then follows that $f$ is
comonotonically modular.

Let us prove that $(ii)\Rightarrow (iii)$. Since $f$ is comonotonically modular and $f_0$ is oddly homogeneous, by Theorem~\ref{thm:wqer87we} we
have that, for every $\sigma\in S_n$ and every $\bfx\in I^n_{\sigma}$,
\begin{eqnarray}
f(\bfx) &=& f(\mathbf{0})+\sum_{1\leqslant i\leqslant p} \big
(f(x_{\sigma(i)}\mathbf{1}_{A_{\sigma}^{\downarrow}(i)})-f(x_{\sigma(i)}\mathbf{1}_{A_{\sigma}^{\downarrow}(i-1)})\big)\nonumber\\
&& \null +\sum_{p+1\leqslant i\leqslant n} \big
(f(x_{\sigma(i)}\mathbf{1}_{A_{\sigma}^{\uparrow}(i)})-f(x_{\sigma(i)}\mathbf{1}_{A_{\sigma}^{\uparrow}(i+1)})\big)\nonumber\\
&=& f(\mathbf{0})+\sum_{1\leqslant i\leqslant p}
\varphi_f(x_{\sigma(i)})\,\big(f(\mathbf{1}_{A_{\sigma}^{\downarrow}(i)})-f(\mathbf{1}_{A_{\sigma}^{\downarrow}(i-1)})\big)\nonumber\\
&& \null +\sum_{p+1\leqslant i\leqslant n}
\varphi_f(x_{\sigma(i)})\,\big(f(\mathbf{1}_{A_{\sigma}^{\uparrow}(i)})-f(\mathbf{1}_{A_{\sigma}^{\uparrow}(i+1)})\big)\label{eq:sa9df78aa}
\end{eqnarray}
for some nondecreasing odd function $\varphi_f\colon I\to\R$, where $p\in\{0,\ldots,n\}$ is such that $x_{\sigma(p)}<0\leqslant
x_{\sigma(p+1)}$. By (\ref{eq:sdfsfd65dsf}), we then obtain $f=\check{L}_{f|_{\B^n}}\circ\varphi_f$. Finally, by (\ref{eq:sa9df78aa}) we then
have that, for every $A\subseteq [n]$,
$$
f_0(\mathbf{1}_A)=\varphi_f(1)f_0(\mathbf{1}_A).
$$
Since there exists $A\subseteq [n]$ such that $f_0(\mathbf{1}_A)\neq 0$ (for otherwise we would have $f_0\equiv 0$ by (\ref{eq:sa9df78aa})), we
obtain $\varphi_f(1)=1$.

The implication $(iii)\Rightarrow (i)$ follows from Proposition~\ref{lemma:ds8686sym}.
\end{proof}

Assume again that $I$ is centered at $0$ with $[-1,1]\subseteq I$ and let $f\colon I^n\to\R$ be a symmetric quasi-Lov\'asz extension for which
there exists $A^*\subseteq [n]$ such that $f_0(\mathbf{1}_{A^*})\neq 0$. Then the inner function $\varphi_f$ introduced in
Theorem~\ref{thm:CanonicalSymQuasi-Lovasz} is unique. Indeed, by Proposition~\ref{lemma:ds8686sym}, we have
$f_0(x\mathbf{1}_A)=\varphi_f(x)f_0(\mathbf{1}_A)$ for every $x\in I$ and every $A\subseteq [n]$. The function $\varphi_f$ is then defined by
$$
\varphi_f(x)=\frac{f_0(x\mathbf{1}_{A^*})}{f_0(\mathbf{1}_{A^*})}~,\qquad x\in I.
$$

We can now describe the possible factorizations of $f$ into compositions of symmetric Lov\'asz extensions with nondecreasing odd functions. The
proof is similar to that of Theorem~\ref{thm:6sad8f65} and thus it is omitted.

\begin{theorem}\label{thm:6sa2d845f65}
Assume that $I$ is centered at $0$ with $[-1,1]\subseteq I$ and let $f\colon\ I^n\to\R$ be a symmetric quasi-Lov\'asz extension,
$f=\check{L}\circ\varphi$. Then there exists $A^*\subseteq [n]$ such that $f_0(\mathbf{1}_{A^*})\neq 0$ if and only if there exists $a>0$ such
that $\varphi=a\,\varphi_f$ and $\check{L}_0=\frac 1a(\check{L}_{f|_{\B^n}})_0$.
\end{theorem}

\begin{remark}
If $I=[-1,1]$, then the ``nonconstant'' assumption and the second condition in assertion $(i)$ of Theorem~\ref{thm:CanonicalSymQuasi-Lovasz} can
be dropped off.
\end{remark}

\section{Application: Quasi-polynomial functions on chains}

In this section we show that prominent classes of lattice functions on closed real intervals are comonotonically modular. To this extent we need to introduce some basic concepts and terminology.

Let $L$ be a bounded distributive lattice. Recall that a \emph{lattice polynomial function on $L$} is a mapping $p\colon L^n\to L$ which can be expressed as combinations of variables and constants using the lattice operations $\wedge$ and $\vee$.
As it is well known, the notion of lattice polynomial function generalizes that of the discrete Sugeno integral. For further background on lattice polynomial functions and discrete Sugeno integrals see, e.g., \cite{CouMar1,CouMar2,CouMar0,Goo67}; see also \cite{BurSan81,Grae03,Rud01} for general background on lattice theory.

In \cite{CouMar09} the authors introduced the notion of ``quasi-polynomial function'' as being a mapping $f\colon X^n\to X$ defined and valued on a bounded chain $X$ and which can be factorized into a composition of a lattice polynomial function with a nondecreasing function.

In the current paper we restrict ourselves to such mappings on closed intervals $J\subseteq\overline{\R}=\left[-\infty,+\infty\right]$. More precisely, by a \emph{quasi-polynomial function on $J$} we mean a mapping $f\colon J^n\to\R$ which can be factorized as
$$
f=p\circ\varphi,
$$
where $\varphi\colon J\to\overline{\R}$ is an order-preserving map and  $p\colon\overline{\R}^n\to\overline{\R}$ is a lattice polynomial function on $\overline{\R}$. For further extensions and generalizations, see \cite{CouMar10,CW1,CW2}.

The class of quasi-polynomial functions was axiomatized in \cite{CouMar09} in terms of two well-known conditions in aggregation theory, which we now briefly describe.

A function $f\colon J^n\rightarrow\overline{\R}$ is said to be \emph{comonotonically maxitive} if, for any two comonotonic tuples $\bfx,\bfx'\in
J^n$,
$$
f(\bfx\vee\bfx') = f(\bfx)\vee f(\bfx').
$$
Dually, $f\colon J^{n}\rightarrow\overline{\R}$ is said to be \emph{comonotonically minitive} if, for any two comonotonic tuples $\bfx,\bfx'\in
J^n$,
$$
f(\bfx\wedge\bfx') = f(\bfx)\wedge f(\bfx').
$$

\begin{theorem}[{\cite{CouMar09,CouMar10}}]\label{thm:CoMinMax}
A function $f\colon J^n\rightarrow\overline{\R}$ is a quasi-polynomial function if and only if it is is comonotonically maxitive and comonotonically minitive.
\end{theorem}

Immediately from Theorem \ref{thm:CoMinMax} it follows that every quasi-polynomial function $f\colon J^n\rightarrow \R$ is comonotonically
modular. Indeed, by comonotonic maxitivity and comonotonic minitivity, we have that, for any two comonotonic tuples $\bfx,\bfx'\in J^n$,
\begin{eqnarray*}
f(\bfx\wedge\bfx')+f(\bfx\vee\bfx') &=& \big( f(\bfx)\wedge f(\bfx')\big)+\big(f(\bfx)\vee f(\bfx')\big)\\
&=& f(\bfx)+f(\bfx').
\end{eqnarray*}
In fact, from Corollary \ref{cor:8sa67}, we obtain the following factorization of quasi-polynomial functions into a sum of unary mappings.

\begin{corollary}
Every quasi-polynomial function $f\colon J^{n}\rightarrow\overline{\R}$ is comonotonically modular. Moreover, for every $\sigma\in S_n$, there exist functions $f^{\sigma}_i\colon J\to\overline{\R}$, $i\in [n]$, such that
$$
f(\bfx)=\sum_{i=1}^nf^{\sigma}_i(x_{\sigma(i)})=\sum_{i=1}^nf^{\sigma}_{\sigma^{-1}(i)}(x_i),\qquad \bfx\in J^n\cap\overline{\R}^n_{\sigma}.
$$
\end{corollary}

\section*{Acknowledgments}

This research is supported by the internal research project F1R-MTH-PUL-09MRDO of the University of Luxembourg.

\end{document}